\documentclass[11pt]{article}
\usepackage[dvips]{epsfig}
\usepackage{amsmath,amssymb,euscript,graphicx,amsfonts,verbatim}
\usepackage{enumerate,color}
\usepackage{graphicx}
\usepackage{enumitem}
\usepackage{enumerate,color}
\usepackage{graphicx}

\usepackage[normalem]{ulem}
\newcommand\redout{\bgroup\markoverwith
{\textcolor{red}{\rule[.5ex]{2pt}{0.4pt}}}\ULon}
\usepackage{xcolor}

\usepackage[colorlinks=true,linkcolor=blue,urlcolor=blue,citecolor=blue]{hyperref}

\newtheorem{theorem}{Theorem}[section]
\newtheorem{proposition}[theorem]{Proposition}
\newtheorem{lemma}[theorem]{Lemma}
\newtheorem{corollary}[theorem]{Corollary}

\newtheorem{definition}[theorem]{Definition}

\newtheorem{notation}[theorem]{Notation}

\newenvironment{proof}{{\noindent \sc Proof.}}{\hfill $\Qed$\\}

\newcommand{\ths}{\theta^*}

\newcommand{\Qed}{\rule{2.5mm}{3mm}}

\newcommand{\Ga}{\Gamma}

\newcommand{\RR}{\mathbb{R}}

\DeclareMathOperator{\mat}{Mat}

\newcommand{\MX}{\mat_X(\RR)}
\newcommand{\e}{E^{*}}

\newcounter{case}

\renewcommand{\thecase}{\arabic{case}}

\newcounter{subcase}

\numberwithin{subcase}{case}

\begin{document}


\begin{center}
{\bf\Large On the $Q$-polynomial property of the full bipartite graph of a Hamming graph} \\ [+4ex]
Blas Fernández{\small$^{a,b}$},  
Roghayeh Maleki{\small$^{a,b}$},  
\v Stefko Miklavi\v c{\small$^{a, b,c}$},
\\
Giusy Monzillo{\small$^{a, b}$} 
\\ [+2ex]
{\it \small 
$^a$University of Primorska, UP IAM, Muzejski trg 2, 6000 Koper, Slovenia\\
$^b$University of Primorska, UP FAMNIT, Glagolja\v ska 8, 6000 Koper, Slovenia\\
$^c$IMFM, Jadranska 19, 1000 Ljubljana, Slovenia}
\end{center}

\begin{abstract}
The $Q$-polynomial property is an algebraic property of distance-regular graphs, that was introduced by Delsarte in his study of coding theory. Many distance-regular graphs admit the $Q$-polynomial property. Only recently the $Q$-polynomial property has been generalized to  graphs that are not necessarily distance-regular. In \cite{attenuated} it was shown that graphs arising from the Hasse diagrams of the so-called attenuated space posets are $Q$-polynomial. These posets could be viewed as $q$-analogs of the Hamming posets, which were not studied in \cite{attenuated}. The main goal of this paper is to fill this gap by showing that the graphs arising from the Hasse diagrams of the Hamming posets are $Q$-polynomial. 
\end{abstract}
\begin{quotation}

\end{quotation}

\begin{quotation}
\noindent 
{\em Math. Subj. Class.:}  05E99, 05C50.
\end{quotation}


\section{Introduction}
\label{sec:intro}

Distance-regular graphs represent a significant class of finite, undirected, connected graphs within Algebraic Combinatorics, a field extensively explored in references \cite{bannai-ito, bannai2021algebraic, BCN, DKT}. These graphs exhibit a remarkable combinatorial regularity that can be analyzed using various algebraic methodologies, including linear algebra techniques such as exploring eigenvalues/eigenvectors of the adjacency matrix \cite[Section 4.1]{BCN}, tridiagonal pairs \cite{ITT}; geometric approaches like linear programming bounds \cite{Del} and root systems \cite[Chapter 3]{BCN}; special functions such as orthogonal polynomials \cite{Leo,ter4},  and hypergeometric series \cite[Chapter 3]{bannai-ito}; and representation theory, where structures like the subconstituent algebra \cite{Terpart1, TerpartII, Terpart3, ter2} and the $q$-Onsager algebra \cite{Ito,IT} play pivotal roles.

Many distance-regular graphs admit an algebraic property called the $Q$-polynomial property.
This property was introduced by Delsarte in his seminal work on coding theory \cite{Del}, and extensively investigated thereafter \cite{bannai-ito, bannai2021algebraic, BCN, DKT, Terpart1, TerpartII, Terpart3}. 

Let $\Ga$ be a $Q$-polynomial distance-regular graph. For each vertex $x$ of $\Ga$ there exists a certain diagonal matrix $A^{*}= A^{*}(x)$, known as the {\em dual adjacency matrix of} $\Ga$ {\em with respect to} $x$. 
The eigenspaces of $A^{*}$ are the subconstituents of $\Ga$ with respect to $x$. The adjacency matrix $A$ of $\Ga$ and the dual adjacency matrix $A^{*}$ are related as follows: the matrix $A$ acts on the eigenspaces of  $ A^{*}$ in a (block) tridiagonal fashion, and $A^*$ acts on the eigenspaces of $A$ likewise \cite[Section 13]{ter2}. 
In \cite{projective},  Terwilliger used this property of $Q$-polynomial distance-regular graphs to  extend the $Q$-polynomial property to graphs that are not necessarily distance-regular. 
To do this,  (i) he dropped the assumption that $\Ga$ is distance-regular; (ii) he dropped the assumption that every vertex of $\Ga$ has a dual adjacency matrix (instead he required that one distinguished vertex of $\Ga$ has a dual adjacency matrix); (iii) he replaced the adjacency matrix of $\Ga$ by a weighted adjacency matrix. 

As the $Q$-polynomial property of general graphs has only been defined recently, there are not a lot of examples of non-distance-regular graphs that are $Q$-polynomial. To our best knowledge, there are only two (infinite) families of $Q$-polynomial graphs, available in the literature, that are not distance-regular.  The first one is a family of graphs associated with the projective geometry $L_N(q)$, see \cite{projective}. The second one is a family of graphs arising from the Hasse diagrams of the so-called {\it Attenuated Space posets} ${\cal A}_q(N,M)$ \cite{attenuated}. The attenuated space posets could be viewed as $q$-analogs of the Hamming posets, which were not studied in \cite{attenuated}. The main goal of this paper is to fill this gap by showing that the Hasse diagram of the Hamming poset (viewed as an undirected graph) is $Q$-polynomial. To describe our result we first give an alternative definition of these graphs. 

Let $\Ga$ denote a connected graph with vertex set $X$ and edge set ${\mathcal E}$. Fix a vertex $x \in X$, and define $\mathcal{E}_f = \mathcal{E} \setminus \{yz \mid \partial(x,y) = \partial(x,z)\}$. Observe that the graph $\Ga_f=(X,\mathcal{E}_f)$ is bipartite. The graph $\Ga_f $ is called the {\em full bipartite graph of} $\Ga$ {\em with respect to} $x$. Let $D, n\geq2$ denote positive integers and $S$ be a set with $n$ elements. The {\em Hamming graph} $H(D,n)$ has vertex set $S^D$ (that is, the set of the ordered $D$-tuples of elements of $S$, or sequences of length $D$ from $S$). Two vertices are adjacent if and only if they differ in precisely one coordinate; that is, if their Hamming distance is one. The main result of this paper is that the full bipartite graph of the Hamming graph $H(D,n)$ is $Q$-polynomial. 

General ideas of our proofs follow ideas from \cite{attenuated}. In particular, a key feature that allows us to provide a $Q$-polynomial structure for $H(D,n)_f$ is a certain equation involving adjacency matrix $A$ of $H(D,n)_f$  and a certain diagonal matrix $A^*$ (which will later turn out to be a dual adjacency matrix of $H(D,n)_f$) - see \cite[Proposition 8.6]{attenuated} and Theorem \ref{thm:cubic} of this paper. However to prove this result we utilize a different tool - we extensively use the fact, that $H(D,n)_f$ is uniform in the sense of Terwilliger \cite{OldTer}. In our future paper we plan to generalize this approach to an arbitrary uniform graph.  To prove that $H(D,n)_f$ is $Q$-polynomial, we also compute the eigenvalues of  $H(D,n)_f$, together with their multiplicities, which is a result of independent interest. As already mentioned above, there are only two (infinite) families of $Q$-polynomial graphs, available in the literature, that are not distance-regular. Therefore, at this stage of research, providing additional examples of such graphs is also very relevant, and so this is yet another contribution of this paper.

Our paper is organized as follows. We recall preliminary definitions and results in Sections \ref{sec:ter}, \ref{sec:lfr}, and \ref{sec:ham}. In Section \ref{sec:ham}, we also define our candidate for the dual adjacency matrix $A^*$ of the full bipartite graph of $H(D,n)$. We prove that the adjacency matrix $A$ of the full bipartite graph of $H(D,n)$ and the matrix $A^*$ satisfy a certain equation (which is cubic in $A$ and linear in $A^*$) in Section \ref{sec:cubic}. We then compute the eigenvalues of $A$ in Section \ref{sec:eig}. In Section \ref{sec:main}, we prove our main result.


\section{Terwilliger algebras and $Q$-polynomial property}
\label{sec:ter}

Throughout this paper, all graphs will be finite, simple, connected, and undirected. Let $\Ga=(X,\mathcal{E})$ denote a graph with vertex set $X$ and edge set $\mathcal{E}$. In this section, we recall the $Q$-polynomial property of $\Ga$ and the definition of a Terwilliger algebra. Let $\partial$ denote the path-length distance function of $\Ga$. The diameter $D$ of $\Ga$ is defined as $D=\max\{\partial(x,y) \mid x,y \in X \}$. Let $\MX$ denote the matrix algebra over the real numbers $\RR$, consisting of all matrices whose rows and columns are indexed by $X$. Let $I \in \MX$ denote the identity matrix and $V$ denotes the vector space over $\RR$ consisting of all column vectors whose coordinates are indexed by $X$. We observe that $\MX$ acts on $V$ by left multiplication. We call $V$ the \emph{standard module}. For any $y \in X$, let $\widehat{y}$ denote the element of $V$ with $1$ in the ${y}$-coordinate and $0$ in all other coordinates. We observe that $\{\widehat{y}\;|\;y \in X\}$ is a basis for $V$. 

\begin{definition}
	\label{def:wam}
	By a weighted adjacency matrix of $\Ga$ we mean a matrix $A \in \MX$ that has $(z,y)$-entry given by
\begin{eqnarray*}
	(A)_{z y} = \left\{ \begin{array}{lll}
		\ne 0 & \hbox{if } \; \partial(z,y)=1, \\
		0 & \hbox{if } \; \partial(z,y) \ne 1 \end{array} \right. 
	\qquad (y \in X).
\end{eqnarray*}
\end{definition}
For the rest of this section, we fix a weighted adjacency matrix $A$ of $\Ga$ that is diagonalizable over $\RR$. Let $M$ denote the subalgebra of $\MX$ generated by $A$. The algebra $M$ is called the \emph{adjacency algebra} of the graph $\Gamma$, generated by $A$. Observe that $M$ is commutative. Let ${\cal D}+1$ denote the dimension of the vector space $M$. Since $A$ is diagonalizable, the vector space $M$ has a
basis $\{E_i\}_{i=0}^{\cal D}$ such that $\sum_{i=0}^{\cal D} E_i=I$ and $E_i E_j = \delta_{i,j} E_i$ for $0 \le i, j \le {\cal D}$. We call $\{E_i\}_{i=0}^{\cal D}$ \emph{the primitive idempotents of} $A$. Since $A \in M$, there exist real numbers $\{\theta_i\}_{i=0}^{\cal D}$ such that $A = \sum_{i=0}^{\cal D} \theta_i E_i$. The scalars $\{\theta_i\}_{i=0}^{\cal D}$ are mutually distinct since $A$ generates $M$. We have
$A E_i = \theta_i E_i = E_i A$ for $0 \le i \le {\cal D}$. Note that
$$
  V = \sum_{i=0}^{\cal D} E_i V \qquad \qquad \text{(direct sum)}.
$$
For $0 \le  i \le  {\cal D}$  the subspace $E_i V$ is an eigenspace of $A$, and $\theta_i$ is the corresponding eigenvalue. For notational convenience, we assume $E_{i} = 0$ for $i < 0$ and $i > {\cal D}$.

Next we discuss the dual adjacency algebras of $\Ga$. To do that, we fix a vertex $x \in X$ for the rest of this section. Let $\varepsilon=\varepsilon(x)$ denote the eccentricity of $x$, that is, $\varepsilon = \max \{\partial(x,y) \mid y \in X\}$. For $ 0 \le i \le \varepsilon$, let $E_i^*=E_i^*(x)$ denote the diagonal matrix in $\MX$ with $(y,y)$-entry given by
\begin{eqnarray*}
	\label{den0}
	(\e_i)_{y y} = \left\{ \begin{array}{lll}
		1 & \hbox{if } \; \partial(x,y)=i, \\
		0 & \hbox{if } \; \partial(x,y) \ne i \end{array} \right. 
	\qquad (y \in X).
\end{eqnarray*}
We call $\e_i$ the \emph{$i$-th dual idempotent} of $\Gamma$ with respect to $x$ \cite[p.~378]{Terpart1}. We observe 
(ei)  $\sum_{i=0}^\varepsilon E_i^*=I$; 
(eii) $E_i^{*\top} = E_i^*$ $(0 \le i \le \varepsilon)$; 
(eiii) $E_i^*E_j^* = \delta_{ij}E_i^* $ $(0 \le i,j \le \varepsilon)$,
where $I$ denotes the identity matrix of $\MX$. It follows that $\{E_i^*\}_{i=0}^\varepsilon$ is a basis for a commutative subalgebra $M^*=M^*(x)$ of $\MX$. The algebra $M^*$ is called the \emph{dual adjacency algebra} of $\Gamma$ with respect to $x$ \cite[p.~378]{Terpart1}. Note that for $0 \le i \le \varepsilon$ we have
$\e_i V = {\rm Span} \{ \widehat{y} \mid y \in X, \partial(x,y)=i\}$, 
and  
\begin{equation}
	\label{vsub}
	V = E_0^*V + E_1^*V + \cdots + E_\varepsilon^*V \qquad \qquad {\rm (direct\ sum}). \nonumber 
\end{equation}
The subspace $\e_i V$ is known as the \emph{$i$-th subconstituent of $\Gamma$ with respect to $x$}. 
For convenience we set $\e_{i}=0$ for $i<0$ and $i>\varepsilon$.
By the triangle inequality, for adjacent $y, z \in X$ the distances $\partial(x, y)$ and $\partial(x, z)$
differ by at most one. Consequently,
$$
  AE^*_iV \subseteq E^*_{i-1} V + E^*_i V + E^*_{i+1} V  \qquad (0 \le i \le \varepsilon). 
$$

Next we discuss the $Q$-polynomial property.

\begin{definition}
	\label{def:q1}
	{\rm (See \cite[Definition 20.6]{ter2}.)} 
	A matrix $A^* \in \MX$ is called a dual adjacency matrix of $\Ga$ (with respect to $x$ and the ordering 
	$\{E_i\}_{i=0}^{\cal D}$ of the primitive idempotents) whenever $A^*$ generates $M^*$ and
	$$
	  A^* E_i V \subseteq E_{i-1} V + E_i V + E_{i+1} V \qquad (0 \le i \le {\cal D}). 
	$$
\end{definition}

\begin{definition}
	\label{def:q2}
	{\rm (See \cite[Definition 20.7]{ter2}.) }
		We say that the ordering $\{E_i\}_{i=0}^{\cal D}$ is $Q$-polynomial
		with respect to $x$ whenever there exists a dual adjacency matrix of $\Ga$ with respect to $x$ and
		$\{E_i\}_{i=0}^{\cal D}$.
	\end{definition}

\begin{definition}
	\label{def:q3}
	{\rm (See \cite[Definition 20.8]{ter2}.) }
		We say that $A$ is $Q$-polynomial with respect to
		$x$ whenever there exists an ordering of the primitive idempotents of $A$ that is $Q$-polynomial
		with respect to $x$.
	\end{definition}

Assume that $\Ga$ has a dual adjacency matrix $A^*$ with respect to $x$ and $\{E_i\}_{i=0}^{\cal D}$. Since $A^* \in M^*$, there exist real numbers $\{\ths_i\}_{i=0}^\varepsilon$ such that $A^* = \sum_{i=0}^\varepsilon \ths_i E^*_i$. The scalars $\{\ths_i\}_{i=0}^\varepsilon$ are mutually distinct since $A^*$ generates $M^*$. We have $A^* E^*_i = \ths_i E^*_i = E^*_i A^*$ for $0 \le i \le \varepsilon$. 
We mentioned earlier that the sum $V = \sum_{i=0}^\varepsilon E^*_i V$ is direct. For $0 
\le i \le \varepsilon$ the subspace $E^*_i V$ is an eigenspace of $A^*$, and $\ths_i$ is the corresponding eigenvalue. As we investigate the $Q$-polynomial property, it is helpful to bring in the Terwilliger algebra
\cite{Terpart1,TerpartII,Terpart3}. The following definition is a variation on \cite[Definition 3.3]{Terpart1}.

The \emph{Terwilliger algebra of $\Gamma$ with respect to $x$ and $A$}, denoted by $T=T(x,A)$,  is the  subalgebra of $\MX$ generated by $M$ and $M^*$. Observe that $T$ is generated by the weighted adjacency matrix $A$ and the dual idempotents $E^*_i \, (0\leq i\leq \varepsilon)$, and hence it is finite-dimensional. If $\Ga$ has a dual adjacency matrix $A^*$ with respect to $x$ and $\{E_i\}_{i=0}^{\cal D}$, then $T$ is generated by $A$ and $A^*$, see \cite[Lemma 2.6]{projective}. The following result will also be useful.

\begin{lemma}
	\label{lem:products}
	{\rm (See \cite[Lemma 2.7]{projective}.)}
	We have $E^*_i  A E^*_j = 0$  if $|i - j| > 1 \; (0 \le i, j \le \varepsilon)$. 
	Assume that $\Ga$ has a dual adjacency matrix $A^*$ with respect to $x$ and $\{E_i\}_{i=0}^{\cal D}$. Then, $E_i A^* E_j = 0$ if $|i - j| > 1 \;
(0 \le i, j \le {\cal D})$.
\end{lemma}

By a \emph{$T$-module} we mean a subspace $W$ of $V$ such that $BW \subseteq W$ for every $B \in T$. Let $W$ denote a $T$-module. Then, $W$ is said to be {\em irreducible} whenever it is nonzero and contains no submodules other than $0$ and $W$.

Let $W$ denote an irreducible $T$-module. Observe that $W$ is a direct sum of the non-vanishing subspaces $E_i^*W$ for $0 \leq i \leq \varepsilon$. The \emph{endpoint} of $W$ is defined as $r:=r(W)=\min \{i \mid 0 \le i\le \varepsilon, \; \e_i W \ne 0 \}$, and the \emph{diameter} of $W$ as $d:=d(W)=\left|\{i \mid 0 \le i\le \varepsilon, \; \e_i W \ne 0 \} \right|-1 $. It turns out that $\e_iW \neq 0$ if and only if $r \leq i \leq r+d$ $(0 \leq i \leq \varepsilon)$, see \cite[Lemma 2.9]{projective}. The module $W$ is said to be \emph{thin} whenever $\dim(E^*_iW)\leq1$ for $0 \leq i \leq \varepsilon$. We say that two $T$-modules $W$ and $W^{\prime}$ are \emph{$T$-isomorphic} (or simply \emph{isomorphic}) whenever there exists a vector space isomorphism $\sigma: W \rightarrow W^{\prime}$ such that $\left( \sigma B - B\sigma \right) W=0$ for all $B \in T$.


\section{The lowering, flat, and raising matrices}
\label{sec:lfr}

Recall our graph $\Ga=(X, {\cal E})$ from Section \ref{sec:ter}. For the rest of this paper, we will assume that $A$ is the (usual) adjacency matrix of $\Ga$: the $(z,y)$-entry of $A$ is equal to 1 if $y,z$ are adjacent, and is equal to 0 otherwise. Fix a vertex $x$ of $\Ga$,  and let $\varepsilon$ denote the eccentricity of $x$. Let $\e_i \; (0 \le i \le \varepsilon)$ be the dual idempotents with respect to $x$, and let $T=T(x,A)$ denote the corresponding Terwilliger algebra. We now recall certain matrices in $T$.

\begin{definition} \label{def2} 
With reference to the notation above, define $L=L(x)$, $F=F(x)$, and $R=R(x)$ in $\MX$ by
	\begin{eqnarray}\label{defLR}
		L=\sum_{i=1}^{\varepsilon}E^*_{i-1}AE^*_i, \hspace{1cm}
		F=\sum_{i=0}^{\varepsilon}E^*_{i}AE^*_i, \hspace{1cm}
		R=\sum_{i=0}^{\varepsilon-1}E^*_{i+1}AE^*_i. \nonumber 
	\end{eqnarray}
	We refer to $L$, $F$, and $R$ as the lowering, the flat, and the raising matrix with respect to $x$, respectively.
\end{definition}
Note that, by definition, $L, F, R \in T$, $F=F^{\top}$, $R=L^{\top}$, and $A=L+F+R$.
Observe that for $y,z \in X$ the $(z,y)$-entry of $L$ equals $1$ if $\partial(z,y)=1$ and $\partial(x,z)= \partial(x,y)-1$ and $0$ otherwise. The $(z,y)$-entry of  $F$ is equal to $1$ if $\partial(z,y)=1$ and $\partial(x,z)= \partial(x,y)$, and is $0$ otherwise. Similarly, the $(z,y)$-entry of $R$ equals $1$ if $\partial(z,y)=1$ and $\partial(x,z)= \partial(x,y)+1$, and is $0$ otherwise. Consequently, for $v \in \e_i V \; (0 \le i \le \varepsilon)$ we have
\begin{equation}
	\label{eq:LRaction}
	L v \in \e_{i-1} V, \qquad  F v \in \e_{i} V, \qquad R v \in \e_{i+1} V.
\end{equation}
Observe also that $\Gamma$ is bipartite if and only if $F=0$. We now recall a connection between matrices $R,F,L$ and certain walks in $\Ga$. The concept of shape of a walk with respect to a fixed vertex was introduced in~\cite{FerMik}. Now, we recall the definition of this concept together with some immediate consequences.
\begin{definition}
	Let $\Gamma = (X , \mathcal{E})$ be a graph. Pick $x, y, z \in  X$, and let
	$P = [y = x_0, x_1, . . . , x_j = z]$ denote a $yz$-walk.  The shape of $P$ with respect to $x$ is a sequence of symbols $t_1 t_2  \ldots t_j$, where for $1 \leq i \leq j$ we have $t_i \in \{\ell, f, r\}$, and such that $t_i = r$ if $\partial(x,x_i) = \partial(x,x_{i-1}) + 1$,  $t_i = f$ if $\partial(x,x_i) = \partial(x,x_{i-1}) $, and $t_i =  \ell$ if $\partial(x,x_i)  = \partial(x,x_{i-1})-1$.
	The number of $yz$-walks of the shape $t_1t_2... t_j $ with respect to $x$ will be denoted as $t_1t_2... t_j(y,z)$, using exponential notation for the shapes containing several consecutive identical symbols.
\end{definition}

The following Lemma is a consequence of elementary matrix multiplication and comments below the Definition \ref{def2} (see also \cite[Lemma 4.2]{FerMik}).

\begin{lemma}
	\label{lem:walks}
	With reference to the notation above, pick $y,z\in X$ and let $m$ be a positive integer. Then, the following hold.
	\begin{enumerate}[label=(\roman*),]
		\item The $(z,y)$-entry of $L^m$ is equal to the number $\ell^m(y,z)$ with respect to $x$.
		\item The $(z,y)$-entry of $L^mR$ is equal to the number $r\ell^m(y,z)$ with respect to $x$.
		\item The $(z,y)$-entry of $RL^m$ is equal to the number $\ell^m r(y,z)$ with respect to $x$.
		\item The $(z,y)$-entry of $LRL$ is equal to the number $\ell r\ell(y,z)$ with respect to $x$.
	\end{enumerate}
\end{lemma}

We finish this section with a definition and a couple of comments.

\begin{definition}\label{def3.1}
With reference to the notation above, define $\mathcal{E}_f = \mathcal{E} \setminus \{yz \mid \partial(x,y) = \partial(x,z)\}$. Observe that the graph $\Ga_f=(X,\mathcal{E}_f)$ is bipartite. The graph $\Ga_f $ is called the full bipartite graph of $\Ga$ with respect to $x$. 
\end{definition}

With reference to Definition \ref{def3.1}, let $\varepsilon=\varepsilon(x)$ denote the eccentricity of $x$ and let $V$ denote the standard module for $\Ga$. Since the vertex set of $\Ga$ is equal to the vertex set of $\Ga_f$, observe that $V$ is also the standard module for $\Ga_f$. Recall that $T$ is generated by the adjacency matrix $A$ and the dual idempotents $\e_i$ ($0\leq i \leq \varepsilon$). Furthermore, we have $A=L+F+R$ where $L,F$, and $R$ are the corresponding lowering, flat, and raising matrices, respectively. Let $A_f$ denote the adjacency matrix of $\Ga_f$ and let $T_f=T_f(x,A_f)$ be the Terwilliger algebra of $\Ga_f$ with respect to $x$. As $\Ga_f$ if bipartite, the flat matrix of $\Ga_f$ with respect to $x$ is equal to the zero matrix. Moreover, the lowering and the raising matrices of $\Ga_f$ with respect to $x$ are equal to $L$ and $R$, respectively. It follows that $A_f=L+R$.  For $0 \le i \le \varepsilon$, note also that the $i$-th dual idempotent of $\Ga_f$ with respect to $x$ is equal to $\e_i$. Consequently, the algebra $T_f$ is generated by the matrices $L, R$, and $\e_i$ ($0\leq i \leq \varepsilon$). Furthermore, the dual adjacency algebra of $\Ga$ with respect to $x$ coincides with the dual adjacency algebra of $\Ga_f$ with respect to $x$.


\section{Distance-regular graphs and Hamming graphs}
\label{sec:ham}

In this section, we review some definitions and basic concepts regarding distance-regular graphs and Hamming graphs.  Let $\Gamma=(X, {\cal E})$ denote a graph with diameter $D$. Recall that for $x,y \in X$ the \emph{distance} between $x$ and $y$, denoted by $\partial(x,y)$, is the length of a shortest $xy$-path. 
For an integer $i$, we define $\Gamma_i(x)$ by 
$$\Gamma_i(x)=\left\lbrace y \in X \mid \partial(x, y)=i\right\rbrace. 
$$
In particular,  $\Gamma(x)=\Gamma_1(x)$ is the set of all neighbors of $x$. For an integer $k \geq 0$, we say that $\Gamma$ is \emph{regular} with valency $k$ whenever $|\Gamma(x)| = k$ for all $x\in X$. 

Fix $x\in X$, and let $\varepsilon$ denote the eccentricity of $x$. Assume that $y\in \Gamma_{i}(x)$ ($0\leq i \leq \varepsilon$). By the triangle inequality we have $\partial(x,z)\in\{i-1,i,i+1\}$. We therefore define the following numbers:
$$
a_i(x,y)= |\Gamma_{i}(x) \cap \Gamma (y)| \quad (0 \le i \le \varepsilon),
$$
$$
b_i(x,y)= |\Gamma_{i+1}(x) \cap \Gamma (y)| \quad (0 \le i \le \varepsilon),
$$
$$
c_i(x,y)= |\Gamma_{i-1}(x) \cap \Gamma (y)| \quad (0 \le i \le \varepsilon).
$$

Note that $b_{\varepsilon}(x,y)=c_0(x,y)=0$. We say that $x$ is \emph{distance-regularized}  (or that $\Gamma$ is \emph{distance-regular around }$x$)  if the numbers $a_i(x,y), b_i(x,y)$, and  $c_i(x,y)$ do not depend on the choice of $y\in \Gamma_{i}(x)$ ($0\leq i \leq \epsilon$). In this case, these numbers are called the \emph{intersection numbers of $x$}, and we abbreviate $a_i(x)=a_i(x,y)$, $b_i(x)=b_i(x,y)$, and $c_i(x)=c_i(x,y)$. 

We say  that $\Gamma$ is \emph{distance-regular} whenever $\Ga$ is distance-regularized around every vertex $z$, and the intersection numbers $a_i(z), b_i(z)$, and $c_i(z)$ do not depend on the choice of $z$. In this case we abbreviate $a_i=a_i(z), b_i=b_i(z)$, and $c_i=c_i(z)$. Note that in a distance-regular graph, the eccentricity of every vertex is equal to $D$. We have $c_i \neq 0 $ for $ 1 \leq i \leq D$ and $b_i \neq 0$ for $0 \leq i \leq D - 1$. Observe that $\Gamma$ is regular with valency $k = b_0$ and $c_i +a_i +b_i = k $ for $ 0 \leq i \leq D$. Moreover, $\Ga$ is bipartite if and only if $a_i=0$ for $0 \le i \le D$. Note that, if $\Ga$ is a distance-regular graph, then its full bipartite graph $\Ga_f=\Ga_f(x)$ is distance-regular around $x$ (but not distance-regular unless $\Ga$ is bipartite).
 
\medskip

We now recall the definition of Hamming graphs. Let $S$ denote a set with
$n\geq2$ elements and let $D$ be a positive integer.
The {\em Hamming graph} $H(D,n)$ has vertex set $S^D$ (that is, the set of the ordered $D$-tuples of elements of $S$, or sequences of length $D$ from $S$). Two vertices are adjacent if and only if they differ in precisely one coordinate; that is, if their Hamming distance is one. The Hamming graph $H(D,n)$ is distance-regular with intersection numbers
$$ 
a_i=i(n-2), \quad b_i=(D-i)(n-1),\quad c_i=i,\qquad (0\leq i\leq D).
$$
Recall that the Hamming graph $H(D, n)$ is bipartite if and only if $n = 2$. For the context of this paper, we will assume that $n \geq 3$. For the rest of this paper, we adopt the following notation.

\begin{notation}
	\label{blank}
	Let $D, n\geq3$ be positive integers and  $H(D,n)$ be the Hamming graph with vertex set $X$. Recall that $H(D,n)$ is distance-regular with intersection numbers $b_i=(D-i)(n-1)$ and $c_i=i$. Fix $x \in X$, and let $\Ga=H(D,n)_f(x)$ denote the full bipartite graph of $H(D,n)$ with respect to $x$.  Let $\e_i \; (0 \le i \le D)$ denote the dual idempotents of $\Ga$ with respect to $x$ and $M^*$ denotes the dual adjacency algebra of $\Ga$ with respect to $x$.  Let $A$ denote the (usual) adjacency matrix of $\Ga$ and $T=T(x,A)$ denote the corresponding Terwilliger algebra. Let $L=L(x)$ and $R=R(x)$ denote the corresponding lowering and raising matrix, respectively. For $0 \le i \le D$ define $\ths_i=D(n-1)-n i$ and let $A^*\in \MX$  be a diagonal matrix defined by
	$$
	(A^*)_{yy}=\theta^*_{\partial(x,y)} \qquad (y \in X).
	$$
\end{notation}

With reference to Notation \ref{blank}, in \cite{FMMM} we showed that $\Ga$ is \emph{uniform} with respect to $x$ in the sense of  \cite{OldTer}. In particular, we have the following Lemma.

\begin{lemma}
	\label{lem:uniform}
	With reference to Notation \ref{blank} we have 
	$$
	  -{1 \over 2} R L^2 + LRL -{1\over 2} L^2 R = (n-1) L.
	$$
\end{lemma}
\begin{proof}
The proof follows immediately from \cite[Theorem 9.5]{FMMM} (or \cite[Theorem 3.2]{OldTer}) and the comments after Definition \ref{def3.1}.
\end{proof}

Moreover, since $\Ga$ is uniform with respect to $x$, the following holds.
\begin{lemma}
	\label{lem:modules}
	{\rm (See \cite[Theorem 2.5 and Theorem 3.3]{OldTer}.)}
	With reference to Notation \ref{blank}, let $W$ denote an irreducible $T$-module with endpoint $r$ and diameter $d$. Then, the following (i)--(iii) hold.
	\begin{itemize}
		\item[(i)] $W$ is thin.
		\item[(ii)] Let $W'$ denote an irreducible $T$-module with endpoint $r'$ and diameter $d'$. Then, $W$ and $W'$ are isomorphic if and only if $r=r'$ and $d=d'$.
		\item[(iii)]  $W$ has a basis $\{w_i\}_{i=r}^{r+d}$ such that the following hold:
		\begin{itemize}
			\item[a)] $w_i \in \e_i W \; (r \le i \le r+d)$; 
			\item[b)] $Lw_r=0$ and $ L w_{r+i} = w_{r+i-1} \; (1\le i \le d)$;
			\item[c)] $Rw_{r+i} = x_{i+1} w_{r+i+1}$, where $x_{i+1}=x_{i+1}(d) =(i+1)(n-1)(d-i) \; (0 \le i \le d-1)$, and $R w_{r+d}=0$.
		\end{itemize}
	\end{itemize}
\end{lemma}
 
With reference to Notation \ref{blank}, the algebra $T$ is semi-simple since the matrix $A$ is symmetric. Consequently, the standard module $V$ is a direct sum of irreducible $T$-modules. Let $W$ be an irreducible $T$-module. By the \emph{multiplicity of} $W$ we mean the number of irreducible $T$-modules in this direct sum, that are isomorphic to $W$. Assume that $W$ has endpoint $r$ and diameter $d$. Since the isomorphism class of $W$ is determined by $r$ and $d$, we will denote the multiplicity of $W$ by ${\rm mult}(r,d)$.

\begin{lemma}
	\label{lem:modules1}
	{\rm (See \cite[Theorem 3.3]{OldTer}.)}
	With reference to Notation \ref{blank}, the following (i), (ii) are equivalent.
	\begin{itemize}
		\item[(i)] There exists an irreducible $T$-module $W$ with the endpoint $r$ and diameter $d$.
		\item[(ii)] $0 \le r \le r+d \le D \le 2r+d$.
	\end{itemize}
    Furthermore, let $0 \le r \le r+d \le D \le 2r+d$. Then, 
    \begin{equation}
    	\label{eq:mult}
    \begin{split}
    	{\rm mult}(r,d) &= {(D+1)_r (n-2)^{2r+d-D}(d+1) \over (D-r-d)! (2r+d-D)! (D+1)} \\
    	&= \frac{d+1}{D-r+1}\binom{D}{2D-2r-d}\binom{2D-2r-d}{D-r-d}(n-2)^{2r+d-D},
    \end{split}
        \end{equation}
    where for a nonnegative integer $a$ we define $(a)_r=1$ if $r=0$ and $(a)_r=a(a-1) \cdots (a-r+1)$ if $r > 0$.
\end{lemma}

\begin{proposition}\label{generates}
	With reference to Notation~\ref{blank}, the matrix $A^*$ generates the dual adjacency algebra $M^*$.
\end{proposition}

\begin{proof}
	It is well-known that $A^*$ is a dual adjacency matrix of the Hamming graph $H(D,n)$, see for example \cite{clebsch}. Therefore, $A^*$ generates the dual adjacency algebra of $H(D,n)$ with respect to $x$. As the dual adjacency algebra of $H(D,n)$ with respect to $x$ coincides with the dual adjacency algebra of $\Ga$ with respect to $x$, the result follows.
\end{proof}


\section{A relationship between the matrices $A$ and $A^*$}
\label{sec:cubic}

With reference to Notation \ref{blank}, in this section, we prove that a so-called {\em tridiagonal relation} \cite[Equation~(132)]{tridiagonal} holds for the pair $(A, A^*)$. In other words, we show that $A, A^*$ satisfy a certain equation, which is cubic in $A$ and linear in $A^*$. Recall from Notation \ref{blank} that we have $\ths_i - \ths_j = n(j-i)$. For $y,z\in X$ we denote by $\gamma_3(y,z)$ the number of walks of length $3$ between $y$ and $z$ in $\Ga$. 

\begin{lemma}
	\label{lem:entries}
	With reference to Notation \ref{blank}, pick $y,z \in X$. Then, the following (i)--(iii) hold.
	\begin{itemize}
		\item[(i)] $( A^3 A^* - A^* A^3)_{zy}=\gamma_3(y,z)(\theta^*_{\partial(x,y)}-\theta^*_{\partial(x,z)})$.
		\item[(ii)] $( A A^* A^2- A^2 A^* A)_{zy}=\displaystyle \sum_{[ y,v,w,z]} ( \theta^*_{\partial(x,w)}-\theta^*_{\partial(x,v)})$, where the sum is over all $3$-walks $[ y,v,w,z]$ between $y$ and $z$.
		\item[(iii)] $( A A^* - A^* A)_{zy}$ is $ ( \theta^*_{\partial(x,y)}-\theta^*_{\partial(x,z)})$ if $\partial(z,y)=1$, and $0$ otherwise. 
	\end{itemize}
\end{lemma}

\begin{proof}
	The equations hold using elementary matrix multiplication, the definition of the matrix $A^*$, and the fact that the $(v,w)$-entry of $A^i$ is equal to the number of walks of length $i$ between $v$ and $w$.
\end{proof}

\begin{theorem}
	\label{thm:cubic}
With reference to Notation \ref{blank}, the pair $(A, A^*)$ satisfies the following tridiagonal relation:
	\begin{equation}\label{eqtrid}
		A^3 A^* - A^* A^3+3( A A^* A^2 - A^2 A^* A)=4(n-1) ( A A^* - A^* A).  
	\end{equation}
\end{theorem}

\begin{proof}
	We will prove that for all $z,y \in X$ the $(z,y)$-entries of both sides of the equation \eqref{eqtrid} agree. It is clear from Lemma \ref{lem:entries} that, if $\partial(z,y) \ge 4$, then the $(z,y)$-entries of both sides of the equation \eqref{eqtrid} equal $0$. Next, if $z=y$ or $\partial(z,y)=2$, then there are no walks of length $3$ between $y$ and $z$ (recall that $\Ga$ is bipartite), and so again, by Lemma \ref{lem:entries}, the $(z,y)$-entries of both sides of the equation \eqref{eqtrid} equal $0$. 
	
	In what follows, we will assume that $\partial(x,z) < \partial(x,y)$; if $\partial(x,z) > \partial(x,y)$, then the proof is analogous. Assume first that $\partial(x,z)=i, \partial(x,y)=i+3$ for some $0 \leq i \leq D-3$, and that $\partial(y,z)=3$. Since $H(D,n)$ is distance-regular and $\Ga$ was obtained from $H(D,n)$ by deleting only the edges that are “flat” with respect to $x$, it follows that all geodesics between $y$ and $z$ in $H(D,n)$ remain geodesics between $y$ and $z$ in $\Ga$. It follows that there are exactly $c_3 c_2 = 6$ walks of length $3$ between $y$ and $z$ in $\Ga$, that is, $\gamma_3(y,z)=6$. Thus, Lemma \ref{lem:entries} implies that 
	$$
	  ( A^3 A^* - A^* A^3)_{zy} = 6(\ths_{i+3}-\ths_i) = -18n,
	$$
	$$
	  ( A A^* A^2- A^2 A^* A)_{zy} = 6(\ths_{i+1}-\ths_{i+2}) = 6n,
	$$
	and $(A A^* - A^* A)_{zy}=0$. We again find that the $(z,y)$-entries of both sides of the equation \eqref{eqtrid} agree. 
	
	It remains to consider the case when $\partial(x,z)=i, \partial(x,y)=i+1$ for some $0 \leq i \leq D-1$, and $\partial(z,y) \in \{1,3\}$. Note that in this case we can have exactly three possible shapes of walks of length $3$ between $y$ and $z$ in $\Ga$, namely walks of shapes $\ell^2 r$, $\ell r \ell$ and $r \ell^2$. Let us abbreviate $a=\ell^2r(y,z)$, $b=\ell r \ell(y,z)$ and $c=r \ell^2(y,z)$. Recall also that by Lemma \ref{lem:walks} we have that $a = (RL^2)_{zy}$, $b=(L R L)_{zy}$ and $c=(L^2 R)_{zy}$, and that by Lemma \ref{lem:uniform} we have 
	$$ 
		-\frac{1}{2}RL^2+LRL -\frac{1}{2} L^2R=(n-1)L.
    $$
    The above comments now yield 
	\begin{eqnarray}\label{entrywise}
		-\frac{1}{2}a+b-\frac{1}{2}c= \left\{ \begin{array}{lll}
			n-1 & \hbox{ if } \; \partial(z,y)=1, \\
			0  &  \hbox{ otherwise. } \end{array} \right. 
	\end{eqnarray}
On the other hand, using Lemma \ref{lem:walks} we find that the $(z,y)$-entry of the left-hand side of \eqref{eqtrid} is equal to 
	\begin{equation}
	\begin{split}
		-n(a+b+c)+3(a(\theta^*_{i-1}-\theta^*_{i})+b(\theta^*_{i+1}-\theta^*_{i})+c(\theta^*_{i+1}-\theta^*_{i+2})) & = \\
		-n(a+b+c) +3(na -nb +nc) = n(2a-4b+2c). &
	\end{split}
	\end{equation}
	Using \eqref{entrywise} we obtain 
	\begin{eqnarray*}
		n(2a-4b+2c)= \left\{ \begin{array}{lll}
			-4n(n-1)  &\hbox{ if } \; \partial(z,y)=1, \\
			0    &\hbox{ otherwise.} \end{array} \right.
	\end{eqnarray*}
	This shows that the $(z,y)$-entries of both sides of the equation \eqref{eqtrid} agree also in this case, thus completing the proof.
\end{proof}


\section{The eigenvalues of $A$}
\label{sec:eig}

With reference to Notation~\ref{blank}, in this section, we compute the eigenvalues of $\Ga$. In order to do that, we first look at the action of the adjacency matrix $A$ on the irreducible $T$-modules.

Recall that by Lemma \ref{lem:modules1}, an irreducible $T$-module with endpoint $r$ and diameter $d$ exists if and only if $0 \le r \le r+d \le D \le 2r+d$. Assume that  $0 \le r \le r+d \le D \le 2r+d$ and let $W$ denote an irreducible $T$-module with endpoint $r$ and diameter $d$. Recall the basis $\{w_i\}_{i=r}^{r+d}$ of $W$ from Lemma \ref{lem:modules}(iii). Let $A_{r,d}$ denote the matrix representing the action of $A$ on $W$ with respect to this basis. Then, by Lemma \ref{lem:modules}(iii) and since $A=R+L$,  we have 

\begin{equation}\label{rep}
A_{r,d}=\begin{pmatrix}
	0 &1                             \\
	x_{1} &0 &1 &  & \text{\huge0}\\
&\ddots  & \ddots & \ddots            \\
	 \text{\huge0} &  &x_{d-1}  & 0  &1    \\
	&  &   &x_{d}   & 0
\end{pmatrix}.
\end{equation}
Let $f_i$ be the characteristic polynomial of the $i$th principal submatrix of $A_{r,d}$  for $1\leq i \leq d+1$. Therefore, $f_{d+1}$ is the characteristic polynomial of $A_{r,d}$. It is straightforward  to check that the polynomials $f_i$ satisfy the recurrence formula
\begin{align*}
	f_{i+1}(t)&=tf_i(t)-x_{i}f_{i-1}(t)=tf_i(t)-i(n-1)(d-i+1)f_{i-1}(t),
\end{align*}
with $f_0(t)=1$, $f_1(t)=t$. The following theorem provides an explicit expression for the polynomial $f_{d+1}$.

\begin{theorem}
	\label{Krat}
	With reference to Notation \ref{blank}, assume that $0 \le r \le r+d \le D \le 2r+d$ and consider the matrix $A_{r,d}$ as defined in \eqref{rep}. Then, the polynomials $f_i \; (0 \le i \le d+1)$, defined above, satisfy
	$$
	  f_i(t)=(2\sqrt{n-1})^iP_i\left(\frac{t}{2\sqrt{n-1}}+\frac{d}{2}\right),
	$$
	where $P_{i}(t)$ is the normalized Krawtchouk polynomial of degree $i$ (see \cite[Equation~(9.11.4)]{koekoekbook} with $p=\frac{1}{2}$ for a definition of these polynomials).
	In particular, the characteristic polynomial of $A_{r,d}$ is equal to
	$$
	  f_{d+1}(t)=(2\sqrt{n-1})^{d+1}P_{d+1}\left(\frac{t}{2\sqrt{n-1}}+\frac{d}{2}\right).
	$$
	
\end{theorem}
\begin{proof}
	By \cite[Equation~(9.11.4)]{koekoekbook},  the normalized
	Krawtchouk polynomials (for  $p=\frac{1}{2}$) satisfy the recurrence formula
	\begin{align}\label{eq2}
		tP_i(t)=P_{i+1}(t)+\frac{d}{2}P_i(t)+\frac{i}{4}(d-i+1)P_{i-1}(t),
	\end{align}
with $P_0(t)=1$ and $P_1(t)=t-d/2$. Define the polynomials $g_i(t)$ as follows
$$
  g_i(t)=(2\sqrt{n-1})^iP_i\left(\frac{t}{2\sqrt{n-1}}+\frac{d}{2}\right).
$$
We will now show that $f_i(t)$ = $g_i(t)$ for $0 \le i \le d+1$. To do that, we first replace $t$ by $\frac{t}{2\sqrt{n-1}}+\frac{d}{2}$ in \eqref{eq2} and multiply the obtained equality by $(2\sqrt{n-1})^{i+1}$ on both sides. Precisely, we get
{\small	\begin{align*}
		\left(\frac{t}{2\sqrt{n-1}}+\frac{d}{2}\right)P_i\left(\frac{t}{2\sqrt{n-1}}+\frac{d}{2}\right)(2\sqrt{n-1})^{i+1}=&P_{i+1}\left(\frac{t}{2\sqrt{n-1}}+\frac{d}{2}\right)(2\sqrt{n-1})^{i+1}\\+&\frac{d}{2}P_i\left(\frac{t}{2\sqrt{n-1}}+\frac{d}{2}\right)(2\sqrt{n-1})^{i+1}\\+&\frac{i}{4}(d-i+1)P_{i-1}\left(\frac{t}{2\sqrt{n-1}}+\frac{d}{2}\right)(2\sqrt{n-1})^{i+1}.
	\end{align*}}
In other words,
$$
  (t+d\sqrt{n-1})g_i(t)=g_{i+1}(t)+d\sqrt{n-1}g_i(t)+i(n-1)(d-i+1)g_{i-1}(t),
$$
or equivalently,
$$
	  tg_i(t)=g_{i+1}(t)+i(n-1)(d-i+1)g_{i-1}(t).
$$
Therefore, the polynomials $f_i$ and $g_i$ satisfy the same recurrence formula. Furthermore, note that 
$$
	g_0(t)=1\qquad\mbox{and} \qquad g_1(t)=(2\sqrt{n-1})P_1\left(\frac{t}{2\sqrt{n-1}}+\frac{d}{2}\right)=2\sqrt{n-1}\frac{t}{2\sqrt{n-1}}=t,
$$
showing that  $f_0(t)=g_0(t)=1$ and $f_1(t)=g_1(t)=t$. This implies that $f_i(t)=g_i(t)$ for all $i$. In particular,
$$
	f_{d+1}(t)=g_{d+1}(t)=(2\sqrt{n-1})^{d+1}P_{d+1}\left(\frac{t}{2\sqrt{n-1}}+\frac{d}{2}\right).
$$
\end{proof}

Consider a linear map on a finite-dimensional vector space. This map is called \emph{multiplicity-free} whenever the map is diagonalizable, and each eigenspace has dimension one.

\begin{corollary}
	\label{eigen}
	With reference to Notation \ref{blank}, assume that $0 \le r \le r+d \le D \le 2r+d$, and consider the matrix $A_{r,d}$ defined in \eqref{rep}. Then, $A_{r,d}$ is multiplicity-free with eigenvalues 
	$$
	  \{\sqrt{n-1}(d-2j)\,\,;\,\,0\leq j\leq d\}.
	$$
\end{corollary}

\begin{proof}
	By Theorem \ref{Krat}, the characteristic polynomial of $A_{r,d}$ is equal to  
	$$
	  f_{d+1}(t)=(2\sqrt{n-1})^{d+1}P_{d+1}\left(\frac{t}{2\sqrt{n-1}}+\frac{d}{2}\right),
	$$
	where $P_{d+1}(t)$ is the normalized Krawtchouk polynomial of degree $d+1$. It is known that the normalized Krawtchouk polynomial of degree $d+1$ factorizes as  $P_{d+1}(x)=x(x-1)\cdots(x-d)$; see for example \cite[p. viii]{koekoekbook}. Pick $0\leq j\leq d$. Then, from Theorem~\ref{Krat} we have that $$f_{d+1}(\sqrt{n-1}(d-2j))=(2\sqrt{n-1})^{d+1}P_{d+1}(d-j)=0.$$
	Therefore, the roots of $f_{d+1}$ are $\{\sqrt{n-1}(d-2j)\,\,;\,\,0\leq j\leq d\}$. As these roots are pairwise different and $A_{r,d}$ has dimension $(d+1) \times (d+1)$, this concludes the proof.
\end{proof}

\begin{theorem}
  \label{thm:eig}
  With reference to Notation \ref{blank}, the matrix $A$ is diagonalizable with eigenvalues $\theta_i \; (0 \le i \le 2D)$, 
  where 
  $$
    \theta_i=\sqrt{n-1}(D-i).
  $$
\end{theorem}

\begin{proof}
	Observe that the matrix $A$ is diagonalizable because the standard module $V$ is a direct sum of irreducible $T$-modules and $A$ is diagonalizable on each irreducible $T$-module. Let $\theta$ denote an eigenvalue of $A$. We first claim that $\theta=\theta_i$ for some $0 \le i \le 2D$. We mentioned that $V$ is a direct sum of irreducible $T$-modules. So, $\theta$ is an eigenvalue for the action of $A$ on some irreducible $T$-module $W$. Let $d$ denote the diameter of $W$. By Corollary \ref{eigen}, we have that $\theta=\sqrt{n-1}(d-2j) = \sqrt{n-1}(D-(D-d+2j))$ for some $0 \le j \le d$. Let us denote $i=D-d+2j$ and observe that $i \ge 0$ since $j \ge 0$ and $d \le D$ by Lemma \ref{lem:modules1}. On the other hand side, we have that $2j \le 2d \le D+d$, which shows that $i=D-d+2j \le 2D$. This proves the claim.

Conversely, we show that $\theta_i$ is an eigenvalue of $A$ for $0 \le i \le 2D$. Let $W_0$ denote the irreducible $T$-module with endpoint $0$ and diameter $D$. Note that this module exists by Lemma \ref{lem:modules1}. By Corollary \ref{eigen} and since the standard module is a direct sum of irreducible $T$-modules, we have that $\theta_i$ with $0 \le i \le 2D$, $i$ even, are eigenvalues of $A$. Similarly, let $W_1$ denote an irreducible $T$-module with endpoint $1$ and diameter $D-1$. Note that this module exists by Lemma \ref{lem:modules1}. By Corollary \ref{eigen}, the eigenvalues of the matrix $A_{1,D-1}$ are $\sqrt{n-1}(D-1-2j) =\sqrt{n-1} (D-(2j+1))\; (0 \le j \le D-1)$, which are exactly the numbers $\theta_i$ with $0 \le i \le 2D$ and $i$ odd. Since the standard module is a direct sum of irreducible $T$-modules, we have that $\theta_i$ with $0 \le i \le D$, $i$ odd, are eigenvalues of $A$. This finishes the proof.  
\end{proof}
 
 \medskip \noindent 
For the sake of completeness, in our next result, we give the multiplicities of the eigenvalues of $A$.
  
  \begin{proposition}
    With reference to Notation \ref{blank} and Theorem \ref{thm:eig}, the following holds for $0 \le i \le 2D$: the multiplicity $m_i$ of the eigenvalue $\theta_i$ is 
    $$
    m_i=\displaystyle \sum_{\substack{\mid D-i \mid\leq d \leq D\\ {d-D+i \ even}}} \ \sum_{0\leq r\leq D-d}^{} {\rm mult}(r,d),
    $$
    where ${\rm mult}(r,d)$ are as given in \eqref{eq:mult}.
  \end{proposition}
  
  \begin{proof}
  	Recall the matrices $A_{r,d}$ defined in \eqref{rep}. Observe that $\theta_i$ is an eigenvalue of $A_{r,d}$ if and only if there exists an integer $j \; (0 \le j \le d)$, such that $D-i=d-2j$, or equivalently, $j=\frac{d-D+i}{2}$. Since the standard module is a direct sum of irreducible $T$-modules and the matrices $A_{r,d} \; (0 \le r \le r+d \le D \le 2r+d)$ are multiplicity-free by Corollary \ref{eigen}, the multiplicity of $\theta_i$ will be equal to 
  	$$
  	  \sum {\rm mult}(r,d),
    $$
    where the sum is over all pairs $(r,d)$ with $0 \le r \le r+d \le D \le 2r+d$ and such that $\theta_i$ is an eigenvalue of $A_{r,d}$. The result follows.
  \end{proof}


\section{A $Q$-polynomial structure for $A$}
\label{sec:main}

With reference to Notation \ref{blank}, in our last section, we prove that the matrix $A^*$ is actually a dual adjacency matrix of $A$ (with respect to certain orderings of the primitive idempotents of $A$). To do that, first recall the eigenvalues 
$$
  \theta_i = \sqrt{n-1}(D-i) \qquad (0 \le i \le 2D)
$$
of the matrix $A$. Let $V_i$ denote the eigenspace of $\theta_i$ and let $E_i$ denote the corresponding primitive idempotent. The following (i)--(iv) are well known.
\begin{itemize}
	\item[(i)] $E_iE_j=\delta_{ij}E_i \qquad (0\leq i,j\leq 2D)$, 
	\item[(ii)] $\displaystyle\sum_{i=0}^{2D} E_i=I$,
	\item[(iii)] $E_iA=\theta_iE_i =A E_i\qquad (0\leq i\leq 2D)$,
	\item[(iv)] $A=\displaystyle\sum_{i=0}^{2D}\theta_iE_i$.
\end{itemize}
Recall also that for notational convenience we assume that $E_i=0$ (and therefore also $V_i=\emptyset$) for $i < 0$ and $i > 2D$.

\begin{proposition}\label{zeroindex}
	With reference to Notation \ref{blank} and the notation above, we have that $E_{i}A^*E_j=0$ for $0\leq i,j\leq 2D$ and $| i-j | \notin\{0,2\}$. 
\end{proposition}

\begin{proof}
	Multiplying the equation \eqref{eqtrid} on the left by $E_i$ and on the right by $E_j$ and using the above-mentioned properties of primitive idempotents, we obtain
	\begin{equation*}
		(\theta_i-\theta_j)  \left((\theta_i-\theta_j)^2-4(n-1)\right)E_iA^*E_j=0.
	\end{equation*}
	Since we have 
	\begin{equation*}
(\theta_i-\theta_j)  \left((\theta_i-\theta_j)^2-4(n-1)\right) = \sqrt{n-1}(n-1)(i-j)(i-j-2)(i-j+2),
	\end{equation*}
the result follows.
\end{proof}

\begin{lemma}\label{tridact}
With reference to Notation \ref{blank} and the notation above, the following holds:
	$$
	  A^* V_i\subseteq V_{i-2}+V_i+V_{i+2}\qquad (0\leq i\leq 2D).
	$$
\end{lemma}
\begin{proof}
	Using the above properties of the primitive idempotents $E_i$ together with the result from Proposition~\ref{zeroindex}, we get
	\begin{align*}
		A^*V_i = A^* E_i V=\sum_{j=0}^{2D}E_j A^* E_i V=& \ E_{i-2}A^*E_iV+E_iA^*E_iV_i+E_{i+2}A^*E_iV_i\\ \subseteq& \ V_{i-2}+V_i+V_{i+2}.
	\end{align*}
\end{proof}
We are now ready to prove our main result.
\begin{theorem}
	\label{thm:main}
	With reference to Notation \ref{blank} and the notation above, the matrix $A^*$ is a dual adjacency matrix of $\Ga$ with respect to $x$ and with respect to the following orderings of primitive idempotents:
	\begin{itemize}
		\item[(i)] $E_0 < E_2 < \cdots < E_{2D} <E_1 <E_3 <\cdots <E_{2D-1}$;
		\item[(ii)] $E_1 < E_3 < \cdots < E_{2D-1} <E_0 <E_2 <\cdots <E_{2D}$.
	\end{itemize}
\end{theorem}
\begin{proof}
The result follows immediately from  Proposition \ref{generates} and Lemma \ref{tridact}.
\end{proof}

\begin{corollary}
With reference to Notation \ref{blank} and the notation above, the adjacency matrix $A$ of  $\Ga$ is $Q$-polynomial.
\end{corollary}

\begin{proof}
	Immediately from Theorem \ref{thm:main}.
\end{proof}

 \end{document}